\documentclass[12pt]{article}

\usepackage{physics} 
\usepackage{enumerate} 
\usepackage{pgfplots}
\usepackage{pgfplotstable}
\usepackage{tikz,pgfplots}
\usepackage{amsmath}  
\usepackage{amssymb}
\usepackage{amsthm}
\usepackage{graphicx}
\usepackage{enumerate}
\usepackage{epsfig}
\usepackage{graphics}
\usepackage{float}
\usepackage{subfigure}
\usepackage{multirow}
\usepackage{lineno}
\usepackage{fullpage}
\usepackage[normalem]{ulem} 
\usepackage{makeidx}
\usepackage{xspace}
\usepackage{wrapfig}
\usepackage{relsize}

\usepackage[cmtip,all]{xy}
\usepackage{mathtools}
\usepackage{tikz-cd}
\usepackage[new]{old-arrows}

\newtheorem{theorem}{Theorem}[section]

\theoremstyle{definition}
\newtheorem{definition}[theorem]{Definition}

\newtheorem{remark}[theorem]{Remark}

\numberwithin{equation}{section}
\allowdisplaybreaks
\makeindex

\usepackage{geometry}
\pgfplotsset{compat=1.14}

\begin{document}

\title{An Extended Trace Formula for Vertex Operators}
\author{Mohammad Reza Rahmati, Gerardo Flores
\footnote{email: mrahmati@cimat.mx (M. R. Rahmati), gflores@cio.mx (G. Flores)}
} 

\date{\today}  

\maketitle  



\begin{abstract} 
We present an extension of the trace of a vertex operator and explain a representation-theoretic interpretation of the trace. Specifically, we consider a twist of the vertex operator with infinitely many Casimir operators and compute its trace as a character formula. To do this, we define the Fock space of infinite level $\mathfrak{F}^{\infty}$. Then, we prove a duality between $\mathfrak{gl}_{\infty}$ and $\mathfrak{a}_{\infty}=\widehat{\mathfrak{gl}}_{\infty}$ of Howe type, which provides a decomposition of $\mathfrak{F}^{\infty}$ into irreducible representations with joint highest weight vector for $\mathfrak{gl}_{\infty}$ and $\mathfrak{a}_{\infty}$. The decomposition of the Fock space $\mathfrak{F}^{\infty}$ into highest weight representations provides a method to calculate and interpret the extended trace. 
\end{abstract}

\section{Introduction}\label{sec:intro}
Characters of representations of infinite-dimensional Lie algebras have been one of the most profound research subjects in the last few decades. Many questions in mathematics and physics can be 
expressed by the representation theory of infinite-dimensional Lie algebras. The character formulas in Lie theory can be interpreted in various concepts with increasing applications. Infinite wedge representation provides examples of the representation of infinite-dimensional Lie algebras that produces such interesting character formulas. The infinite wedge representation is a fundamental discrete structure on which many problems in theoretical quantum physics can be modeled.
It provides a general theoretic representation framework that supports different sampling problems in quantum theory. The infinite wedge representation is sometimes referred to as \textit{the fermionic Fock space}. Interesting character formulas can be extracted from specific operators acting on the Fock space trace that provide a powerful tool to study, for instamce, generating series, Faymann integrals, and probability amplitudes in mathematical physics, \cite{BKY, BO, CL, DMP, DVZ, Ma, Mat, NY1, St, Ve, VK, Wi, Ze}. 

A vertex operator is an operator of an infinite-dimensional Lie algebra that appears in the form of generating a series of operators. Vertex operators present a formalism for the linear action on specific infinite-dimensional vector spaces, such as the fermionic Fock space. The trace of the vertex operators is related to Schur functions and symmetric polynomials. In this context, the symmetric functions play a prominent role in connecting combinatorics to the representation theory of infinite-dimensional Lie algebras. The theory of vertex operators involves the representation theory of infinite symmetric group $S_{\infty}$ and the theory of symmetric polynomials; there are well-known formulas for the trace of vertex operators in terms of symmetric polynomials. 

The computation of the trace of the vertex operators $Tr ( q^{L_0} \exp ( \sum_n A_n \alpha_{-n}) \exp ( \sum_n B_n \alpha_n ) )$, is crucial in representation theory. The transformation beneath the trace is acting on common Fock space $\mathfrak{F}=\mathfrak{F}^1$, also called the infinite wedge representation. One can consider higher Casimir operators $L_j, \ j >0$ acting on the Fock space $\mathfrak{F}$. We follow a computation of Bloch-Okounkov \cite{BO} for the character of the infinite wedge representation, where a product formula is established for the character. Vertex operators appear in the context of string theory partition functions of CY 3-folds. In this case, the vertex operator is twisted by one or more Casimir operators, where the interest is to calculate its trace. We give a natural extension of the trace to the case where infinitely many Casimir operators appear in the trace function. Our idea is to use a decomposition of the Fock space of level infinity into irreducible highest weight representations of Lie superalgebra $\mathfrak{a}_{\infty} =\hat{\mathfrak{gl}}_{\infty}$. The decomposition breaks the trace into a sum of the traces on the irreducible components, \cite{CL, CW}.
\subsection{Contributions}  
Based on the computation of the trace formulas for $Tr ( q^{L_0} \exp ( \sum_n A_n \alpha_{-n}) \exp ( \sum_n B_n \alpha_n ) )$ and the Bloch-Okounkov result on the character formula for $Tr \left ( \exp(\sum_{j\geq 0} 2 \pi i L_j)\right )$, we propose to compute the following trace:
\begin{equation} \label{eq:trace}
Trace=Tr \left ( \exp(\sum_{j\geq 0} 2 \pi i L_j) \exp ( \sum_{n > 0} A_n \alpha_{-n}) \exp ( \sum_{n > 0} B_n \alpha_n ) \right ).
\end{equation}
Our approach is to define the Fock space of level $\infty$ denoted by $\mathfrak{F}^{\infty}$, as the natural generalization of the Fock space of finite level $l$. Then, we show the existence of a decomposition,
\begin{equation} \label{eq:decomposition}
\mathfrak{F}^{\infty}=\bigoplus_{\lambda}L(\mathfrak{gl}_{\infty}, \lambda) \otimes L(\mathfrak{a}_{\infty}, \Lambda(\lambda))
\end{equation}
where $\lambda$ runs overall generalized partitions, from which a character formula can be calculated. The $\lambda$ summand is a joint-highest weight representation of $\mathfrak{gl}_{\infty}$ and $\mathfrak{a}_{\infty}$. We shall interpret the trace in \eqref{eq:trace} as the trace of an operator acting on the Fock space $\mathfrak{F}^{\infty}$. The decomposition above in \eqref{eq:decomposition} gives a way to express the trace formula as a sum of traces over highest weight representations of $\mathfrak{gl}_{\infty} \times \mathfrak{a}_{\infty}$, where we can conduct the computation directly.    
\subsection{Organization of the text} 
The remainder of this paper is as follows. Section \ref{sec:prelim} provides basic definitions on Fock spaces and infinite wedge representation. We also introduce the vertex operators and character formulas for the vertex operators. Section \ref{sec:problem} explains the problem that we are going to solve in the paper together with its motivation from physics. Section \ref{sec:main} contains the main contributions of the paper. Further, we add an application given in Section \ref{sec:app}. Some conclusions are given in Section \ref{sec:concl}. Finally, the appendix contains the definition of several infinite-dimensional Lie algebras that are important in this context. 
\section{Preliminaries} \label{sec:prelim}
\subsection{Infinite wedge representation}
The Fock space is an infinite-dimensional vector space that appears to represent certain infinite-dimensional Lie algebras. It provides a systematic framework to express generating a series of importance in Physics by the trace of vertex operators [acting on the Fock space]. It also plays a crucial role in string theory to explain the probabilistic amplitudes.
\begin{definition} \label{def:infinitewedge}
The half infinite wedge or fermionic Fock space $\mathfrak{F}$ defined by:
\begin{equation}
\mathfrak{F}=\bigwedge^{\frac{\infty}{2}}V= \bigoplus_{i_r \in 1/2 +\mathbb{Z}} \ \mathbb{C}. \ v_{i_1} \wedge v_{i_2} \wedge ..., \qquad i_j=i_{j-1}-1/2, \ j \gg 0,
\end{equation}
is the vector space spanned by the semi-infinite wedge product of a fixed basis of the infinite-dimensional vector space,
\begin{equation}
V =\sum_{i \in 1/2 +\mathbb{Z}}\mathbb{C}.v_i,
\end{equation}
i.e., the monomials $v_{i_1} \wedge v_{i_2} \wedge ...$ such that:
\begin{itemize}
\item $i_1 >i_2>...$
\item $i_j=i_{j-1}-1/2$ for $j \gg 0$.
\end{itemize}
\end{definition}

Besides, we have the creation and annihilation operators defined by:
\begin{equation}
\begin{aligned}
\psi_k:&v_{i_1} \wedge v_{i_2} \wedge ... \mapsto v_k \wedge v_{i_1} \wedge v_{i_2} \wedge ...\\
\psi_k^*:&v_{i_1} \wedge v_{i_2} \wedge ... \mapsto (-1)^l v_{i_1} \wedge v_{i_2} \wedge ...\wedge \widehat{v_{i_l=k}} \wedge ...
\end{aligned}
\end{equation}

The monomials can be parametrized in terms of partition 
\begin{equation}
|\lambda \rangle =v_{\lambda}=\lambda_1-1/2 \wedge \lambda_2 -3/2 \wedge ...
\end{equation}

The Fock space $\mathfrak{F}$ is almost a Hilbert space, with respect to the inner product $\langle \ v_{\lambda}, \ v_{\mu} \ \rangle =\delta_{\lambda \mu}$. Its completion with respect to the norm of the inner product is a Hilbert space. We can also write this using the Frobenius coordinates of partitions:
\begin{equation}
|\lambda \rangle =\prod_{i=1}^l\psi_{a_i}^*\psi_{b_i}|0 \rangle ,\qquad a_i=\lambda_i-i+1/2,\ b_i=\lambda_i^t-i+1/2,
\end{equation}
where $(a_1,...,a_l|b_1,...,b_l)$ are called the Frobenius coordinates of $\lambda$. The operator  
\begin{equation}
C = \sum_{k \in 1/2 + \mathbb{Z}} :\psi_k\psi_k^*:
\end{equation}
is called the charge operator, whose action on $\mathfrak{F}$ is:
\begin{equation}
C \left (v_{i_1} \wedge v_{i_2} \wedge ... \right ) = [(\sharp \  \text{present positive}\  v_i)
-(\sharp \  \text{missing negative}\  v_i)]v_{i_1} \wedge v_{i_2} \wedge ...
\end{equation}

The vectors of $0$-charge are characterized by being annihilated by $C$. Besides, the energy operator $H$ (Hamiltonian) is defined by
\begin{equation}
H = \sum_{k \in 1/2 + \mathbb{Z}} k :\psi_k\psi_k^*:
\end{equation}
and satisfies $Hv_{\lambda}= |\lambda|v_{\lambda}$. Now, let define the Boson operators:
\begin{equation}
\alpha_n=\sum_{k \in 1/2 +\mathbb{Z}} \psi_{k+n}\psi_k^*.
\end{equation}
They satisfy the commutation relations $[\alpha_n, \psi_k]=\psi_{k+n},\ [\alpha_n, \psi_k^*]=-\psi_{k-n}$. 

The fermionic Fock space $\mathfrak{F}$ is the Hilbert space generated by a pair of fermions $\psi^{\pm}(z)$ with components $\psi_r^{\pm}, \ r \in \frac{1}{2}+\mathbb{Z}$ satisfying the Clifford commutation relations,
\begin{equation}
[\psi_i^+,\psi_j^-]=\delta_{i,-j}, \qquad [\psi_i^+,\psi_j^+]=0, \qquad [\psi_i^-,\psi_j^-]=0.
\end{equation}
\subsection{Bosonic Fock space}
The Hilbert space $\mathfrak{F}$ can be constructed in two main forms which are isomorphic (Boson-Fermion correspondence). They are called Fermionic and Bosonic Fock spaces. The definition \ref{def:infinitewedge} is usually referred as the fermionic Fock space. The Bosonic Fock space is a representation defined on the polynomial ring $\mathbb{C}[x_1,x_2, ...;q,q^{-1}]$, see \cite{Ze}. It is not hard to write a specific isomorphism between $\mathfrak{F}$ and the mentioned polynomial ring. In this context, the vertex operators act as cetain differential operators, and the trace of vertex operator appears as generating a series of the ring of symmetric functions on infinitely many variables. We review some features of this below. 

Consider the coordinate ring of the affine variety $Sym^k(\mathbb{C})$, namely $B_k=\mathbb{C}[x_1,...,x_n]/S_n$. Write the Hilbert series of $B_k$ as $H_{B_k}(q)=\sum_n q^n h_n(B_k)$, where $h_n(B_k)=\natural \{\text{monomials in}\ B_k \ \text{of charge} \ k\}$. Let also $q$ act as $\mathbb{C}^{\times}$ on the other variables. We understand that $B_k$ is the ring of symmetric functions in variables $x_1,...,x_k$. It is also generated by Schur functions, i.e., $B_k=\big \langle \ S_{\mu}(x_1,...,x_k), \ |\mu| \leq k \ \big \rangle$. In string theory $B_k$ arises as the Hilbert space $\mathcal{H}_k$ generated by the Boson oscillators up to charge $k$, with commutation relations, $[\alpha_n,\alpha_m]=n\delta_{n+m,0}$. 

Let us associate monomials to partitions as follows:
\begin{equation}
\lambda=1^{\lambda_1}2^{\lambda_2}.... \longmapsto \alpha_{-1}^{-\lambda_1}\alpha_{-2}^{-\lambda_2}....|0 \rangle ,
\end{equation}
thus,
\begin{equation}
B_k \cong \mathcal{H}_k=\left \langle  \ \alpha_{-1}^{m_1}...\alpha_{-k}^{m_k} |0 \rangle \ | \ m_1,...m_k \geq 0 \ \right \rangle,
\end{equation}
and we have the inclusions $\mathcal{H}_0 \subset \mathcal{H}_1 \subset ...$ corresponding to the nested sequence of Young diagrams of increasing number or rows. The $\mathbb{C}^{\times}$ induces an action of $Sym^*(\mathbb{C})$ such that the functions $S_{\mu}(x_1,...,x_k)$ are eigen-functions with eigenvalues $q^{|\mu|}$ of the action of $q^{L_0}$ on $\mathcal{H}$, where 
\begin{equation}
L_0=\sum_{n >0}\alpha_{-n}\alpha_n.
\end{equation}
Then, we can write:
\begin{equation}
H_{B_k}(q)=Tr_{\mathcal{H}_k}q^{L_0}=\sum_{|\mu| \leq k}q^{|\mu|}=\prod_{n=1}^k(1-q^n)^{-1}.
\end{equation}
The generating function of these series becomes:
\begin{equation}
G(t,q)=\sum_{n=0}^{\infty}t^kH(B_k)(q)=\sum_kt^kTr_{\mathcal{H}_k}q^{L_0}=\sum_{\mu}s_{\mu}(t)s_{\mu}(1,q,...).
\end{equation}
We can generalize this argument by considering the coordinate ring of the affine variety $B_{k_1,...,k_n}=Sym^{k_1}(\mathbb{C}) \times ...\times Sym^{k_n}(\mathbb{C})$, and define the generating series $G(t_1,...,t_n;q)=\sum_{k_1,...,k_n}t_1^{k_1}...t_n^{k_n}H_{B_{k_1,...,k_n}}(q)$. The ring $B_{k_1,...,k_n}$ is generated by the product of the Schur polynomials $B_{k_1,...,k_n}=\big \langle S_{\mu_1}(x_{1,1},...,x_{1,k_1}). \  . \dots . \ S_{\mu_n}(x_{n,1},...,x_{n,k_n}); \ l(\mu_j) \leq k_j\big \rangle$. The above ring is isomorphic to the Hilbert space spanned by the Bosonic operators,
\begin{equation}
\mathcal{H}_{k_1,...,k_n}=\left \langle \prod_{j=1}^n (\alpha_{-1}^j)^{n_{j,1}}...(\alpha_{-k_n}^j)^{n_{j,k_n}} |0 \rangle  \right \rangle .
\end{equation} 

In the new terminology, the Hilbert series can be written as:
\begin{equation}
H_{B_{k_1,...,k_n}}(q)=Tr_{\mathcal{H}_{k_1,...,k_n}}q^{L_0},
\end{equation}
where $L_0=\sum_{j=1}^n\sum_{r >0}\alpha_{-r}^j\alpha_r^j$ is the charge operator. We obtain $G(t_1,...,t_n,q)=\prod_{j=1}^nG(t_j,q)=\prod_{j=1}^n\prod_i(1-q^{-i}t_j)$. The classical Boson-Fermion correspondence is an isomorphism between two representations of the Heisenberg algebra, namely \textit{the Bosonic Fock space} and \textit{the fermionic Fock space}. The Boson-Fermion correspondence is a basic result in mathematical physics. There are various applications of this correspondence. It provides an explicit way of comparing expressions for $q$-dimensions of representations, through which new combinatorial identities are derived by computing
characters of representations in two different ways. 
\subsection{The Fock space of level $l$} \label{def:Fock-spacefinite}
The Fock space of level $l$ denoted by $\mathfrak{F}^l$ is the Fermionic Fock spaces on $l$ pairs of fermions $\psi_r^{\pm,j}, r \in 1/2+\mathbb{Z}, \ j=1,...,l$. Let denote $\widehat{C}^l$ be the Clifford algebra on the components. $\mathfrak{F}^l$ is a simple $\widehat{C}^l$-module generated by $|0 \rangle$, such that $\psi_r^{\pm,j}|0 \rangle =0, \ r >0$. The $\frac{1}{2} \mathbb{Z}_+$-gradation of $\mathfrak{F}^l$ is given by the eigenvalues of the degree operator $d$ such that $d|0 \rangle=0, \ [d, \psi_{-r}^{\pm,j}]=r\psi_{-r}^{\pm,j}$, where any graded subspace is finite-dimensional. We shall use normal ordering notations defined by,
\begin{equation}
\begin{aligned}
:\psi_r^{+,j}\psi_s^{-,k}:&=\begin{cases}-\psi_s^{-,k}\psi_r^{+,j}, \qquad s=-r<0\\
\psi_r^{+,j}\psi_s^{-,k}\end{cases}\\
:\psi_r^{-,j}\psi_s^{+,k}:&=\begin{cases}-\psi_s^{+,k}\psi_r^{-,j}, \qquad s=-r<0\\
\psi_r^{-,j}\psi_s^{+,k}\end{cases}\\
:\psi_r^{+,j}\psi_s^{+,k}:&=\psi_r^{+,j}\psi_s^{+,k}\\
:\psi_r^{-,j}\psi_s^{-,k}:&=\psi_r^{-,j}\psi_s^{-,k}\\
:\psi_r^{\pm,j}\phi_s^{+,k}:&=\psi_r^{\pm,j}\phi_s^{+,k}.
\end{aligned}
\end{equation}

Define the operators $e_{ij}^*$ and $e_*^{ij}$ by:
\begin{equation} \label{eq:e-ij}
e_{ij}^*=\sum_{k=1}^l:\psi_{1/2-i}^{+,k}\psi_{j-1/2}^{-,k}:, \qquad 
e_*^{ij}=\sum_{r \in 1/2 +\mathbb{Z}}:\psi_{-r}^{+,i}\psi_r^{-,j}:
\end{equation} 
With the above set-up, we have the following:
\begin{itemize} 
\item The map $\mathfrak{a}_{\infty} \longrightarrow  End(\mathfrak{F}^l), \ E_{ij} \longmapsto e_{ij}^*$ defines a representation of $\mathfrak{a}_{\infty}$. 
\item The map $\mathfrak{gl}_{l} \longrightarrow  End(\mathfrak{F}^l), \ E^{ij} \longmapsto e_*^{ij}$ is a representation of $\mathfrak{gl}_l$.
\item The above actions of $\mathfrak{gl}_{l}$ and $\mathfrak{a}_{\infty}$ on $\mathfrak{F}^l$ commute, [see \cite{CW} sec. 5.4]. Thus, we have the relation  
\begin{equation}
[e_{ij}^*,e_*^{rs}]=[\sum_{u \in 1/2+\mathbb{Z}}:\psi_{-u}^{-,i} \psi_{-u}^{-,j}:\ ,\ \sum_{l=1}^{\infty}:\psi_{1/2-r}^{+,l} \psi_{1/2-s}^{-,l}:]=0 .
\end{equation}
\end{itemize}

We use the generalized partitions as the weights of representations of $\mathfrak{gl}_l$, i. e., the partitions $\lambda$ of the form $\lambda=(\lambda_1 \geq \lambda_2 \geq... \geq \lambda_{i-1} \geq \lambda_i = ...=\lambda_{j-1} =0 \geq \lambda_j \geq ... \geq \lambda_r)$. Also, for the weight of representations of $\mathfrak{a}_{\infty}$ we denote $\Lambda(\lambda)=\Lambda_{\lambda_1}+...+\Lambda_{\lambda_l}=l\Lambda_0+\sum_i\lambda_i'\epsilon_i$, where,
\begin{equation}
\lambda_i'=\begin{cases} \ |\ \{\ j \ |\lambda_j \geq i,\ i \geq 1\ \}\ |\\ 
-|\ \{\ j \ | \lambda_j <i,\ i\leq 0 \ \} \ |\end{cases}.
\end{equation}
Notice that we are using the notation in Appendix A, in particular eq. \eqref{eq:weights-a}.

Let's define,
\begin{equation}
\begin{aligned}
\varpi_m^{+,j}=\psi_{-m+1/2}^{+,s}...\psi_{-3/2}^{+,s}\psi_{-1/2}^{+,s}\\
\varpi_m^{-,s}=\psi_{-m+1/2}^{-,s}...\psi_{-3/2}^{-,s}\psi_{-1/2}^{-,s}
\end{aligned}, \qquad \varpi_0^{+,s}=1,
\end{equation}
A joint highest weight vector in $\mathfrak{F}^l$ associated to $\lambda$ with respect to the standard Borel $\mathfrak{gl}(l)\times \mathfrak{a}_{\infty}$ is 
\begin{equation}
v_{\lambda}=\varpi_{\lambda_1}^{+,s}\varpi_{\lambda_2}^{+,s}...\varpi_{\lambda_i}^{+,s}\varpi_{-\lambda_j}^{-,s}\varpi_{-\lambda_{j+1}}^{+,s}...\varpi_{-\lambda_l}^{+,s},
\end{equation}
whose weight with respect to $\mathfrak{gl}(l)$ is $\lambda$, and with respect to $\mathfrak{a}_{\infty}$ is $\Lambda(\lambda)$, [see \cite{CW} sec. 5.4].  
\begin{theorem} [See \cite{CW} sec. 5.4] \label{thm:finite-decomposition}
We have the decomposition 
\begin{equation}
\mathfrak{F}^l=\bigoplus_{\lambda}L(\mathfrak{gl}(l), \lambda) \otimes L(\mathfrak{a}_{\infty}, \Lambda(\lambda)).
\end{equation}
\end{theorem}
We may state other kind of dualities by considering different vertex operators. In order to illustrate the fermionic operators, let $V=(\mathbb{C}^l \otimes \mathbb{C}^{\infty}) \bigoplus (\mathbb{C}^{l*} \otimes \mathbb{C}^{{\infty}*})$, where $\mathbb{C}^{\infty}$ is a vector space with basis $w_r, r \in -1/2-\mathbb{Z}_+$, the dual indexed is given by $w_{-r}$, and $\mathbb{C}^l$ has basis $v^{+,i}$ with dual $v^{-,i}$. Then, we may illustrate $\psi_r^{\pm,i}=v^{\pm,i} \otimes w_r$.
\subsection{The character of Vertex operators}
The operators of the form:
\begin{equation}
\Gamma_+(x)=\exp(\sum_{n \geq 1}\frac{x^n}{n}\alpha_n), \qquad \Gamma_-(x)=\exp(\sum_{n >0} \frac{x^n}{n}\alpha_{-n})
\end{equation}
are called \textit{vertex operators}. They are adjoint with respect to the natural inner product. We have a commutation relation:
\begin{equation} \label{eq:commutation}
\Gamma_+(x)\Gamma_-(y)=(1-xy)
\Gamma_-(y)\Gamma_+(x).
\end{equation}
Also, we have,
\begin{equation}
\Gamma_+(x) v_{\mu}=\sum_{\lambda \supset \mu}S_{\lambda/\mu}(x)v_{\lambda}.
\end{equation} 

Vertex operators provide powerful tools to express partitions. For example, we can write,
\begin{equation}
\begin{aligned}
\Gamma_+(1) |\mu \rangle =\sum_{\lambda \supset \mu}|\lambda \rangle\\
\Gamma_-(1) |\mu \rangle =\sum_{\lambda \subset \mu}|\lambda \rangle .
\end{aligned}
\end{equation}
Besides, we may write the McMahon function as:
\begin{equation}
\begin{aligned}
Z&=\sum_{\text{3-dim partitions}}q^{\natural \ boxes} = \langle (\prod_{t=0}^{\infty}q^{L_0}\Gamma_+(1)) q^{L_0} (\prod_{t=-\infty}^{-1}\Gamma_-(1)q^{L_0})\rangle\\
&=\langle \prod_{n>0}\Gamma_+(q^{n-1/2})\prod_{n>0}\Gamma_-(q^{-n-1/2}) \rangle.
\end{aligned}
\end{equation}
We may divide a 3-dimensional partition into slices of two-dimensional partitions, for instance, along the diagonals. In this way, the vertex operator divides into the multiplication of many vertex operators of the slices,
\begin{equation}
Z(\{x_m^{\pm}\})=\langle ... \prod_{u_i< m<v_{i+1}}\Gamma_+(x_m^-)\prod_{v_i< m<u_{i}}\Gamma_-(x_m^+)...\rangle = \langle \prod_{u_0< m<u_n}\Gamma_{-\epsilon(m)}(x_m^{\epsilon(m)})  \rangle.
\end{equation}
In this way, one can obtain product formulas such as 
\begin{equation}
Z(\{x_m^{\pm}\})\prod_{m_1 <m_2} (1-x_{m_1}^-x_{m_2}^+).
\end{equation}
For example, we may choose 
\begin{equation}
\begin{aligned}
\{x_m^{\pm}\}&=\{t^iq^{v_i}|i=1,2,...\}\\
\{x_m^{\pm}\}&=\{t^{j-1}q^{-v_i^t}|j=1,2,...\},
\end{aligned}
\end{equation}
and we get,
\begin{equation}
Z_{\lambda\mu\nu}(t,q)=\langle \prod_{u_0< m<u_n}\Gamma_{-\epsilon(m)}(x_m^{\epsilon(m)}) \rangle .
\end{equation}

The partition function can also be read by putting a wall on the distance $M$ along with one of the axis. Then, using commutation \eqref{eq:commutation} we have expressions of the form,
\begin{equation}
\begin{aligned}
Z&=\langle \prod_{0<m <\infty}\Gamma_-(x_m^+)\prod_{-M<m<0}\Gamma_+(x_m^-) \rangle  \\
&=\prod_{l_1=1}^{\infty}\prod_{l_2=1}^M(1-x_{l_1-1/2}^+x_{-l_2+1/2}^-)^{-1}(\langle \prod_{-M<m<0}\Gamma_+(x_m^-)\prod_{0<m<\infty}\Gamma_-(x_m^+) \rangle ,
\end{aligned}
\end{equation}
where the last factor in parentheses is equal to $1$, and we obtain a product formula. 

The following theorem is another instant of generating series which arise from the trace of vertex operator.
\begin{theorem} \cite{Ze}
We have the following formula for the trace of a vertex operator acting on $\mathfrak{F}=\bigwedge^{\frac{\infty}{2}}V$:
\begin{equation}
Tr \left ( q^{L_0} \exp ( \sum_n A_n \alpha_{-n}) \exp ( \sum_n B_n \alpha_n ) \right )= \prod_n \sum_{k}\sum_{l=0}^k\frac{n^lA_n^lB_n^l}{l!}q^{nk}{k \choose l} ,
\end{equation}
where $L_0$ is the charge operator, and $q^{L_0}| \lambda \rangle =|\lambda ||\lambda \rangle$.
\end{theorem}
\begin{proof} 
Denote the operator in the trace by $T$. We have the isomorphism:
\begin{equation}
\bigwedge^{\frac{\infty}{2}}V= \bigotimes_{n=1}^{\infty} \bigoplus_{k=0}^{\infty} \alpha_{-n}^k|0 \rangle  ,
\end{equation}
which implies:
\begin{equation}
\begin{aligned}
Tr(T)&=\prod_{n=1}^{\infty} Tr \left ( T|_{\bigoplus_{k=0}^{\infty} \alpha_{-n}^k|0 \rangle} \right )\\
&= \prod_n \sum_k \langle 0|\alpha_{n}^k \Big | q^{L_0}e^{A_n\alpha_{-n}}e^{B_n\alpha_n} \Big | |\alpha_n^k|0 \rangle \rangle\\
&= \prod_n \sum_{k,l,m}\frac{A_n^lB_n^m}{l!m!}q^{n(l-m+k)} \langle 0|\alpha_{n}^k \Big | \alpha_{-n}^l\alpha_n^m \Big | |\alpha_n^k|0 \rangle \rangle\\
&= \prod_n \sum_{k,l}\frac{A_n^lB_n^l}{l!l!}q^{nk} \langle 0|\alpha_{n}^k \Big | \alpha_{-n}^l\alpha_n^l \Big | |\alpha_n^k|0 \rangle \rangle\\
&= \prod_n \sum_{k}\sum_{l=0}^k\frac{n^lA_n^lB_n^l}{l!}q^{nk}{k \choose l} .
\end{aligned}
\end{equation}
\end{proof}
In the following, we present another calculation due to Bloch, and Okounkov \cite{BO} of the trace of a representation on the infinite wedge space defined in \ref{def:infinitewedge}. This is the base of our main result \ref{th:main} given in Section \ref{sec:main}.
\begin{theorem} \cite{BO}
Let $\mathfrak{F}=\bigwedge^{\frac{\infty}{2}}V$ be the Fock space on a fixed basis of $V=\bigoplus_{j \in \mathbb{Z}} \mathbb{C}v_j$. Then, the character of $\mathfrak{F}$ is given by
\begin{equation}
ch(\mathfrak{F})= \prod_{n \geq 0} (1+y_0y_1^{n+1/2}y_2^{n+3/2)^2}....)(1+y_0^{-1}y_1^{n-1/2}y_2^{-(n-1/2)^2}....),
\end{equation}
where $y_j=e^{2\pi i \tau_j}$.
\end{theorem}
\begin{proof}
The associated character is the trace of the operator,
\begin{equation} 
y_0^{L_0}y_1^{L_1}....y_n^{L_n}...= \exp(\sum_j 2 \pi i \tau_j L_j) ,
\end{equation}
where $L_i$ are commuting operators which act on $\mathfrak{F}$ by 
\begin{equation}
L_j \longmapsto \sum_{n \in 1/2 +\mathbb{Z}}(n-\frac{1}{2})^jE_{n,n}.
\end{equation}
The operators $L_j$ mutually define a representation of
\begin{equation}
H=\mathbb{C}L_0 \oplus \mathbb{C}L_1 \oplus ... \longrightarrow End(\mathfrak{F}).
\end{equation}
We compute the action of the operators $\exp(2\pi i\tau_j L_j)$ on the basis elements. Thus, we have,
\begin{equation}
\begin{aligned}
L_j. \left (\psi_{-i_1}\psi_{-i_2}...\psi_{-i_l}\psi_{-j_1}^*...\psi_{-j_k}^*|0 \rangle \right )=  ( \sum_{a=1}^l (i_a-1/2)^j-&\sum_{b=1}^k (-j_b-1/2)^j) 
\\ &\left (\psi_{-i_1}\psi_{-i_2}...\psi_{-i_l}\psi_{-j_1}^*...\psi_{-j_k}^*|0 \rangle \right ) .
\end{aligned}
\end{equation}
We also have,
\begin{equation}
\begin{aligned}
\exp(2 \pi i \tau_jL_j)\psi_{-n}|0 \rangle&=\left ( y_0y_1^{n-1/2}y_2^{(n-1/2)^2}...\right ) \psi_{-n}|0 \rangle \\
\exp(2 \pi i \tau_jL_j)\psi_{-n}^*|0 \rangle&=\left ( y_0^{-1}y_1^{n-1/2}y_2^{-(n-1/2)^2}... \right ) \psi_{-n}^*|0 \rangle .
\end{aligned} 
\end{equation}
The claim of the theorem follows from the isomorphism:
\begin{equation}
\mathfrak{F}=\bigwedge^{\frac{\infty}{2}}V=\bigotimes_n (1+\psi_{-n})(1+\psi_{-n}^*)|0 \rangle .
\end{equation}
\end{proof}
\section{Problem statement} \label{sec:problem}
Motivated by the two calculations of the trace of vertex operator presented in the previous section, i.e., 
\begin{equation} 
Tr \left ( q^{L_0} \exp ( \sum_n A_n \alpha_{-n}) \exp ( \sum_n B_n \alpha_n )\right ), \qquad Tr \left ( \exp(\sum_{j\geq 0} 2 \pi i L_j)\right ),
\end{equation} 
where $\alpha_n, \ n \in \mathbb{Z}$ are Boson operators; $L_0$ is the energy operator; and $L_j, \ j >0$ are certain Casimir operators, \cite{Ze, BO}. Thus, we propose to compute the following trace:
\begin{equation} \label{eq:infinite-trace}
Trace=Tr \left ( \exp(\sum_{j\geq 0} 2 \pi i L_j) \exp ( \sum_{n > 0} A_n \alpha_{-n}) \exp ( \sum_{n > 0} B_n \alpha_n ) \right ).
\end{equation}
Therefore, we pose the following questions:
\begin{itemize}
\item How can one compute the trace in terms of the former traces?
\item What is the representation theory interpretation of that?
\item In case that the coefficients $A_n, B_n$ are suitably chosen, what is the Physical interpretation of the trace in terms of string theory partition functions?
\end{itemize}

The character can be studied from different points of view. A direct way to calculate it could be to expand the exponentials inside the trace, apply basic commutation rules between the operators, and then compute the matrix elements. Some formulas in Lie theory, such as the Backer-Campbell-Hausdorff formula or the Wick formula, can also be helpful for the calculation. Although this method can bring computational insights toward the above question, it hits Adhoc complexity and difficulties. One may expand the operators in the trace both in terms of Bosonic operators $\alpha_{\pm n}$, and also fermionic operators $\psi_j, \ j \in \frac{1}{2}+\mathbb{Z}$.  
\section{Main results} \label{sec:main}
In order to interpret and compute a trace formula for \eqref{eq:infinite-trace} we make the following definition that is a natural generalization of the Fock space of level $l$, defined in \ref{def:Fock-spacefinite}.
\begin{definition} (Fock space $\mathfrak{F}^{\infty}$) \label{def:Fock-spaceinfinite} 
Consider the fermionic fields $\psi_r^{\pm,j}, \ r \in 1/2+\mathbb{Z}, \ j \in \mathbb{Z}$ satisfying the natural Clifford commutation relations,
\begin{equation}
\begin{aligned}
[\psi_r^{+,i}, \psi_s^{-,j}]&=\delta_{i,j}\delta_{r,-s}I \\
[\psi_r^{+,i},\psi_r^{+,j}]&=[\psi_r^{-,i}, \psi_s^{-,i}]=0 .
\end{aligned}
\end{equation}
Set also $\widehat{C}^{\infty}$ be the Clifford algebra on these fields. By definition $\mathfrak{F}^{\infty}$ is a simple $\widehat{C}^{\infty}$-module generated by $|0 \rangle$, such that $\psi_r^{\pm,j}|0 \rangle =0, \ r >0$. 
\end{definition}
Next, we express a duality of Howe-type for the pair $(\mathfrak{gl}_{\infty}, \mathfrak{a}_{\infty})$. In other words, the Fock space $\mathfrak{F}^{\infty}$ is a representation of both the Lie algebras $\mathfrak{gl}_{\infty}$ and, $\mathfrak{a}_{\infty}$. Moreover, $\mathfrak{F}^{\infty}$ decomposes to the sum of their irreducible representation.
Next, we are ready to present our first main result.
\begin{theorem} \label{thm:howe-duality}(Main Result. ($\mathfrak{gl}_{\infty}, \mathfrak{a}_{\infty})$-Howe duality) \label{th:duality-infinite} There exists a decomposition,
\begin{equation} \label{eq:dec-duality}
\mathfrak{F}^{\infty}=\bigoplus_{\lambda}L(\mathfrak{gl}_{\infty}, \lambda) \otimes L(\mathfrak{a}_{\infty}, \Lambda(\lambda)) ,
\end{equation}
where $\lambda$ runs over all generalized partitions. Besides, there is a character formula,
\begin{equation}
ch(\mathfrak{F}^{\infty})=\prod_i^{\infty}\prod_j^{\infty} (1+y_jx_i)(1+y_j^{-1}x_i^{-1}),
\end{equation}
where $x_i, y_j,\ i,j \in \mathbb{N}$ are variables.
\end{theorem}
\begin{proof} 
In \eqref{eq:e-ij}, we replace the operators $e_{ij}^*$ by,
\begin{equation}
e_{ij}^*=\sum_{k=-\infty}^{\infty}:\psi_{1/2-i}^{+,k}\psi_{j-1/2}^{-,k}: \ , \qquad i,j \in \mathbb{Z}.
\end{equation}
The map $\mathfrak{a}_{\infty} \longrightarrow  End(\mathfrak{F}^{\infty}), \  E_{ij} \longmapsto e_{ij}^*$ defines a representation of $\mathfrak{a}_{\infty}$. Also, let define the operators $e_*^{ij}(n)$ by
\begin{equation}
e_*^{ij}=\sum_{r \in 1/2 +\mathbb{Z}}:\psi_{-r}^{+,i}\psi_r^{-,j}: , \qquad (r \in 1/2 +\mathbb{Z}, \ i,j \in \mathbb{Z}).
\end{equation} 
The map $\mathfrak{gl}_{\infty} \longrightarrow  End(\mathfrak{F}^{\infty}), \ E^{ij} \longmapsto e_*^{ij}$ is a representation of $\mathfrak{gl}_{\infty}$. We need to check that the action of $\mathfrak{gl}_{\infty}$ and $\mathfrak{a}_{\infty}$ commute. That is,
\begin{equation}
[e_{ij}^*,e_*^{rs}]=[\sum_{u \in 1/2+\mathbb{Z}}:\psi_{-u}^{-,i} \psi_{-u}^{-,j}:, \sum_{l=1}^{\infty}:\psi_{1/2-r}^{+,l} \psi_{1/2-s}^{-,l}:]=0 .
\end{equation}
A joint highest weight vector in $\mathfrak{F}^{\infty}$, associated to a generalized partition $\lambda=(\lambda_1,...,\lambda_j)$ w.r.t the standard Borel of $\mathfrak{gl}_{\infty} \times \mathfrak{a}_{\infty}$, is $v_{\lambda}=\varpi_{\lambda_1}^{+,s}\varpi_{\lambda_2}^{+,s}...\varpi_{\lambda_i}^{+,s}\varpi_{-\lambda_j}^{-,s}\varpi_{-\lambda_{j+1}}^{+,s}...\varpi_{-\lambda_l}^{+,s}$. Where whose weight w.r.t. $\mathfrak{gl}_{\infty}$ is $\lambda$ and w.r.t. $\mathfrak{a}_{\infty}$ is $\Lambda(\lambda)$. By applying any root vector of $\mathfrak{gl}_{\infty}$ and $\mathfrak{a}_{\infty}$ to $v_{\lambda}$, it produces two identical $\psi_{\bullet}^{\bullet,\bullet}$ in the resulting monomial. As in the finite case any irreducible representation of $\mathfrak{gl}_{\infty}$ appears in $\mathfrak{F}^{\infty}$ and the multiplicity free decomposition $\mathfrak{F}^{\infty}=\bigoplus_{\lambda}L(\mathfrak{gl}_{\infty}, \lambda) \otimes L(\mathfrak{a}_{\infty}, \Lambda(\lambda))
$ follows.

The character of $\mathfrak{gl}_{\infty}$ on $\mathfrak{F}^{\infty}$ is the trace of the operator $\prod_ix_i^{\epsilon_*^{ii}}$. The character of $\mathfrak{a}_{\infty}$ on $\mathfrak{F}^{\infty}$ is the trace of $\prod_j y_j^{\epsilon_{jj}^*}$. Therefore, $ch(\mathfrak{F}^{\infty})$ is the trace of the product of the two operators. Calculating the trace of the operator $\prod_ix_i^{\epsilon_*^{ii}}\prod_j y_j^{\epsilon_{jj}^*}$ on both sides of \eqref{eq:dec-duality} we get:
\begin{equation}
ch(\mathfrak{F}^{\infty})=\bigoplus_{\lambda}ch \left ( L(\mathfrak{gl}_{\infty}, \lambda) \right ) \otimes ch \left ( L(\mathfrak{a}_{\infty}, \Lambda(\lambda))\right )=\prod_i\prod_j (1+y_jx_i)(1+y_j^{-1}x_i^{-1}).
\end{equation}
\end{proof}
The duality in Theorem \ref{th:duality-infinite} enables us to compute the trace formula \eqref{eq:infinite-trace} by computing it on each summand. Specifically, we have the following.
\begin{theorem}(Main result) \label{th:main}
We have the following formula for the trace in \eqref{eq:infinite-trace}
\begin{equation} \label{eq:trace-formula}
\begin{aligned}
Trace=\sum_{\lambda} \prod_{n \geq 0} (1+y_0y_1^{n+1/2}y_2^{n+3/2)^2}....)(1+y_0^{-1}y_1^{n-1/2}y_2^{-(n-1/2)^2}....)\\
\prod_{r \geq 1}y_r^{p_r(\lambda)}\sum_{\mu\prec \lambda} S_{\lambda/\mu}^{(A_n)}(x_1, x_2, ...)S_{\lambda^t/\mu}^{(B_n)}(x_1, x_2,...),
\end{aligned}
\end{equation}
where $x_i$ and $y_i$ are independent variables, and 
\begin{equation}
p_r(\lambda)=\sum_l(\lambda_l-l+\frac{1}{2})^r+(-1)^{r+1}(l-\frac{1}{2})^r=\sum_l(m_l+\frac{1}{2})^r+(-1)^{r+1}(n_l+\frac{1}{2})^r
\end{equation}
holds, where $(m_1,...,m_s|n_1,...,n_s)$ are Frobenius coordinates of $\lambda$. The effect of the coefficients $A_n , B_n$ is absorbed in the variables $x_1, x_2,...,$, [see the proof for the explanation on dependence to the coefficients $A_n, B_n$].
\end{theorem}
\begin{remark}
The dependence of the above trace to the coefficients $A_n$ and $B_n$ is somehow formal. The contribution to the trace coming from $Tr \left ( \exp ( \sum_{n > 0} A_n \alpha_{-n}) \exp ( \sum_{n > 0} B_n \alpha_n ) \right )$ is given in the last sum appearing in \eqref{eq:trace-formula}, where the effect of $A_n, B_n$ is absorbed in the variables $x_1, x_2, ...$, [see \cite{Ze} page 11 and 24, or \cite{Ma} pages 25 and 70 for the notation].
\end{remark}
\begin{proof} (proof of Theorem \ref{th:main})
Let us denote,
\begin{equation}
\mathcal{L}= \exp(\sum_j 2 \pi i L_j), \qquad  
T= \exp ( \sum_n A_n \alpha_{-n}) \exp ( \sum_n B_n \alpha_n ).
\end{equation}
According to Theorem \ref{thm:howe-duality}, we need to compute:
\begin{equation}
\sum_{\lambda} 
Tr \left ( \mathcal{L}\big |_{L(\mathfrak{gl}_{\infty}, \lambda)} \right ) 
Tr \left ( T \big |_{L(\mathfrak{a}_{\infty}, \Lambda^{\mathfrak{a}_{\infty}}(\lambda)} \right ).
\end{equation}
We first compute the factor relevant to $\mathcal{L}$. Consider $v_{\lambda}=| \lambda \rangle $, the vector of weight $\lambda$. We have the formula:
\begin{equation}
\mathfrak{F}^{\lambda}=\bigotimes_n (1+\psi_{-n})(1+\psi_{-n}^*)|\lambda \rangle .
\end{equation}
By lemma 5.1 in \cite{BO}, we also have:
\begin{equation}
\begin{aligned}
\exp(2 \pi i L_j)\psi_{-n}|\lambda \rangle&=\left ( y_0y_1^{n-1/2}y_2^{(n-1/2)^2}... \prod_{r \geq 1}y_r^{p_r(\lambda)}\right ) \psi_{-n}|\lambda \rangle \\
\exp(2 \pi i L_j)\psi_{-n}^*|\lambda \rangle&=\left ( y_0^{-1}y_1^{n-1/2}y_2^{-(n-1/2)^2}...\prod_{r \geq 1}y_r^{p_r(\lambda)}\right ) \psi_{-n}^*|\lambda \rangle ,
\end{aligned} 
\end{equation}
therefore 
\begin{equation}
Tr \left ( \mathcal{L}\big |_{L(\mathfrak{gl}_{\infty}, \lambda)} \right )=\prod_{n \geq 0} (1+y_0y_1^{n+1/2}y_2^{n+3/2)^2}....)(1+y_0^{-1}y_1^{n-1/2}y_2^{-(n-1/2)^2}....)\prod_{r \geq 1}y_r^{p_r(\lambda)}.
\end{equation}

On the other hand, it is well known that,
\begin{equation} 
Tr \left ( T \big |_{L(\mathfrak{a}_{\infty}, \Lambda^{\mathfrak{a}_{\infty}}(\lambda)} \right )=\sum_{\mu\prec \lambda} S_{\lambda/\mu}^{(A_n)}(x_1, x_2, ...)S_{\lambda^t/\mu}^{(B_n)}(x_1, x_2,...) ,
\end{equation}
where $S_{\lambda/\mu}^{(A_n)}$ is the skew Schur function $S_{\lambda/\mu}$ specialized to the case in which the symmetric power functions $p_n$
equal $nA_n$. Similarly, $S_{\lambda/\mu}^{(B_n)}$ is the skew Schur function $S_{\lambda/\mu}$ specialized to the case in which the symmetric power functions $p_n$ equals $nB_n$, [see \cite{Ze} page 11 and 24 for the notation].  
\end{proof}
\section{Application}\label{sec:app}
The string theory partition function of toric $CY$ $3$-folds can be formulated from their tropical diagram by basic combinatorial rules of the topological vertex, 
\cite{HIKLV}, \cite{IKa1}, \cite{IK}, \cite{IKS1}, \cite{AKV1}. These partition functions can also be described combinatorially by certain vector fields acting on the tropical diagram that fix the vertices. The action of a symmetry group preserves the action, a unitary group $U(N)$, called the gauge group. In this context, one computes another generating series related to the topological one. The equality of these two series has been shown in many cases. In this way, one may consider tropical diagrams that are more complicated, especially when there are many cells. According to the known formulations of the generating series under consideration, this motivates more twists in the corresponding vertex operator. Thus one may ask what happens if we apply infinitely many twists in the vertex operator. Although finding the corresponding geometric object that produces such a tropical configuration may raise the question, we expect that the associated diagram can be explained according to some limiting procedure in the real world. Next, we present an instance of how this can be presented. The partition function of $U(1)$ theory can be written in the form: 
\begin{equation}\label{eq:CYpartition}
Z(\tau, m, \epsilon)=Tr \left ( Q_{\tau}^{L_0} \exp \left (\sum_{n \geq 1}\frac{Q_m^n-1}{n(q^{n/2}-q^{-n/2})}\alpha_n \right ) \exp \left (\sum_{n \geq 1}\frac{Q_{-m}^n-1}{n(q^{n/2}-q^{-n/2})}\alpha_{-n} \right ) \right ) .
\end{equation}
Using the commutation relation of $\alpha_{\pm}$ it can be written as follows, cf. \cite{IK, IKS1, AKV1},
\begin{equation}
Z(\tau, m, \epsilon)=\prod_k(1-Q_{\tau}^k)^{-1}\prod_{i,j}\frac{(1-Q_{\tau}^kQ_m^{-1}q^{i+j-1})(1-Q_{\tau}^kQ_mq^{i+j-1})}{(1-Q_{\tau}^kq^{i+j-1})}.
\end{equation}

Besides, the partition function in \eqref{eq:CYpartition} can be generalized to
\begin{equation} \label{eq:CYextendedtrace} 
Z(\tau, m, \epsilon,t)=Tr \left ( Q_{\tau}^{L_0} e^{\sum_nt_nL_n} \exp \left (\sum_{n \geq 1}\frac{Q_m^n-1}{n(q^{n/2}-q^{-n/2})}\alpha_n \right ) \exp \left (\sum_{n \geq 1}\frac{Q_{-m}^n-1}{n(q^{n/2}-q^{-n/2})}\alpha_{-n} \right ) \right ) .
\end{equation}
Also, in the limit $m \mapsto 0$ we obtain,
\begin{equation}
Z(\tau, m=0, \epsilon,t)=Tr \left ( Q_{\tau}^{L_0} e^{\sum_nt_nL_n} \right ).
\end{equation}

We can write the partition function in terms of the Gromov-Witten potentials: 
\begin{equation}
Z(\tau, m, \epsilon)=\exp\left (\sum_{g \geq 0}\epsilon^{2g-2}F_g \right ),
\end{equation}
where 
\begin{equation}
e^{F_1}=\prod_k(1-Q_{\tau}^k)^{-1}\left ( \frac{(1-Q_{\tau}^k)^{2}Q_m^{-1})(1-Q_{\tau}^kQ_m)^{2}}{(1-Q_{\tau}^k)^{4}} \right )^{1/24}.
\end{equation}
A natural motivation is to investigate the interpretation of the trace in \ref{eq:CYextendedtrace}. One way to proceed with this is to think of \eqref{eq:CYextendedtrace} as a limit of the case when finitely many $L_j$ exists in the trace function. The case with finitely many $L_j$ naturally appears in the topological string partition functions associated to toric CY 3-folds where the web diagram has an $M \times N$ cell structure, where $M,N \in \mathbb{N}$.
\section{Conclusions}\label{sec:concl}
An extension of the calculation of the trace of the vertex operator with infinitely many Casimirs is presented based on a duality of Howe type for the pair $(\mathfrak{a}_{\infty}, \mathfrak{gl}_{\infty})$. The trace formula applies to an extension of the topological string theory partition function of CY 3-folds.
\section*{Appendix: The Lie algebras $\mathfrak{a}_{\infty}, \mathfrak{c}_{\infty}$ and $\mathfrak{d}_{\infty}$}
We introduce the three infinite-dimensional Lie algebras whose representations are crucial in string theory. The reference for this appendix is \cite{CW} Section 5.4, where all the materials discussed can be found there with more details. 

\vspace{0.3cm}
\noindent 
\textbf{(1) Lie Algebra $\mathfrak{a}_{\infty}$ :} Let $\mathfrak{a}_{\infty}=\widehat{\mathfrak{gl}}_{\infty}=\mathfrak{gl}_{\infty} \oplus \mathbb{C}K$ with the braket,
\begin{equation} 
[X+cK,Y+dK]=[X,Y]'+\tau(X,Y)K.
\end{equation}
The function $\tau(X,Y)=Tr([J,X]Y)$ is called a cocycle function, where $J=\sum_{j \leq 0}E_{jj}$ and $[.,.]'$ is the bracket of $\mathfrak{gl}_{\infty}$. We have the degree gradation  $\mathfrak{gl}_{\infty}=\bigoplus_j\mathfrak{gl}_{\infty}^j$, where $j$ runs over integers and it is called $\mathbb{Z}$-principal gradation. The degree of $E_{ij}$ would be $j-i$. Besides, we have a decomposition: 
\begin{equation}
\mathfrak{a}_{\infty} =\mathfrak{a}_{\infty}^+ \oplus \mathfrak{a}_{\infty}^0 \oplus \mathfrak{a}_{\infty}^-, \qquad \mathfrak{a}_{\infty}^{\pm}= \bigoplus_{j>0} \mathfrak{gl}_{\infty}^{\pm j}, \ \mathfrak{a}_{\infty}^0=\mathfrak{gl}_{\infty}^0 \oplus \mathbb{C}K.
\end{equation} 
The root system of $\mathfrak{a}_{\infty}$ is 
\begin{equation}
R=\{\epsilon_i-\epsilon_j\ |i \ne j \}, \qquad \epsilon_i(E_{ii})=\delta_{ij},\  \epsilon_i(K)=0.
\end{equation}
The set $\Pi=\{\epsilon_i-\epsilon_{i+1}\ |i \in \mathbb{Z}\}$ is a fundamental system for $\mathfrak{a}_{\infty}$ with corresponding co-roots $\{H_i^{\mathfrak{a}}=E_{ii}-E_{i+1,i+1}+\delta_{i,0}K\}$. We denote by $\Lambda_j^{\mathfrak{a}}$ the $j$-th fundamental weight of $\mathfrak{a}_{\infty}$, i.e. $\Lambda_j^{\mathfrak{a}}(H_i^{\mathfrak{a}})=\delta_{ij}, \ (i \in \mathbb{Z}), \ \Lambda_j^{\mathfrak{a}}(K)=1$. A straight forward computation gives: 
\begin{equation}\label{eq:weights-a}
\begin{aligned}
\Lambda_j^{\mathfrak{a}}&= \Lambda_0^{\mathfrak{a}}-\sum_{i=j+1}^0 \epsilon_i \qquad j <0 \\ 
\Lambda_j^{\mathfrak{a}}&= \Lambda_0^{\mathfrak{a}}+\sum_{i=1}^j \epsilon_i \qquad \ \ j \geq 1.
\end{aligned}
\end{equation}
The irreducible highest weight representation of $\mathfrak{a}_{\infty}$ of the highest weight $\Lambda$ is denoted by $L(\mathfrak{a}_{\infty}, \Lambda)$. 
 
\vspace{0.3cm}
\noindent
\textbf{(2) Lie algebra $\mathfrak{c}_{\infty}$:} Let $V=\bigoplus_{j \in \mathbb{Z}}\mathbb{C}v_j$ be the vector space generated by the vectors $v_j$, where $E_{ij}v_j=v_i$. Consider the symmetric bilinear form,
\begin{equation}
C(v_i,v_j)=(-1)^i\delta_{i,1-j}, \forall i,j \in \mathbb{Z}.
\end{equation} 
Set $\mathfrak{c}_{\infty}=\bar{\mathfrak{c}}_{\infty} \oplus \mathbb{C}K$, where 
\begin{equation} 
\bar{\mathfrak{c}}_{\infty}=\{X \in \mathfrak{gl}_{\infty}\ |\ C(X(u),v)+C(u,X(v))=0\}.
\end{equation}
A fundamental system for $\mathfrak{c}_{\infty}$ is given by $\{-2\epsilon_1, \epsilon_i-\epsilon_{i+1}; \ i \geq 0\}$ with simple co-roots:
\begin{equation}
\begin{aligned} 
H_i^{\mathfrak{c}}&=E_{ii}+E_{-i,-i}=E_{i+1,i+1}-E_{1-i,1-i}, \\
H_0^{\mathfrak{c}}&=E_{00}-E_{11}+K.
\end{aligned}
\end{equation} 
The $j$-th fundamental weight for $\mathfrak{c}_{\infty}$ is  
defined by the same as the case for $\mathfrak{a}_{\infty}$ and explicitly written as follows,
\begin{equation}
\Lambda_j^{\mathfrak{c}}=\Lambda_0^{\mathfrak{c}}+\sum_{i=1}^j \epsilon_i, \qquad j \geq 1.
\end{equation}

\vspace{0.3cm}
\noindent 
\textbf{(3) Lie Algebra $\mathfrak{d}_{\infty}$:} Define the Lie algebra $\mathfrak{d}_{\infty}=\bar{\mathfrak{d}}_{\infty} \oplus \mathbb{C}K$ where,
\begin{equation}
\bar{\mathfrak{d}}_{\infty}=\{X \in \mathfrak{gl}_{\infty}|D(X(u),v)=D(u,X(v))\},
\end{equation}
and where $D(v_i,v_j)=\delta_{i,1-j}$. It has the fundamental system $\{\pm\epsilon_1-\epsilon_2, \epsilon_i-\epsilon_{i+1}, \ i \geq 2\}$ with simple co-roots: 
\begin{equation}
\begin{aligned} 
H_i&=E_{ii}-E_{-i,-i}-E_{i+1,i+1}-E_{1-i,1-i}, \\
H_0&=E_{00}-E_{-1,-1}-E_{22}-E_{11}+2K.
\end{aligned}
\end{equation}


\end{document}